\documentclass[a4paper,11pt,reqno]{amsart}
\usepackage[english]{babel}
\usepackage{amsmath}
\usepackage{amscd}
\usepackage{amsgen}
\usepackage{amsfonts}
\usepackage{amssymb}
\usepackage{amsthm}
\usepackage{comment}
\usepackage{stmaryrd}
\usepackage{amsmath}
\usepackage[all]{xy}
\usepackage{latexsym}
\usepackage{graphics}
\usepackage{verbatim}
\usepackage{url}
\usepackage{slashed}

\newtheorem{theorem}{Theorem}
\newtheorem{prop}[theorem]{Proposition}
\newtheorem{lem}[theorem]{Lemma}
\newtheorem{cor}[theorem]{Corollary}

\newcommand{\lo}{\mathring{L}}                   
\newcommand{\Hess}{\operatorname{Hess}}
\newcommand{\tr}{\operatorname{tr}}
\def\J{{\sf J}}       
\newcommand{\id}{\operatorname{Id}}         
\def\st{\stackrel{\text{def}}{=}}
\newcommand{\JF}{\mathcal{F}}                      

\newcommand{\W}{\overline{\mathcal{W}}}        
\newcommand{\Ric}{\operatorname{Ric}}
\def\Rho{{\sf P}} 
          
\newcommand{\scal}{\operatorname{scal}}

\newcommand{\G}{{\mathcal G}}                   

\numberwithin{theorem}{section} \numberwithin{equation}{section}

\title{Extrinsic Paneitz operators and $Q$-curvatures for hypersurfaces}

\author[Andreas Juhl]{Andreas Juhl}

\address{Humboldt-Universit\"at, Institut f\"ur Mathematik, Unter den Linden 6, 10099 Berlin, Germany}

\email{ajuhl@math.hu-berlin.de}
\email{juhl.andreas@googlemail.com}


\begin{document}

\begin{abstract}
For any hypersurface of a Riemannian manifold, recent works introduced the notions of extrinsic conformal 
Laplacians and extrinsic $Q$-curvatures. Here we announce explicit formulas for the extrinsic Paneitz operators ${\bf P}_4$ 
and the corresponding extrinsic $Q$-curvatures for totally umbilic hypersurfaces in any dimension. Moreover, we state 
explicit formulas for the critical ${\bf P}_4$ and the total integral of the critical ${\bf Q}_4$ in the general case. This integral 
is a global conformal invariant. Finally, we establish an analog of the Alexakis-Deser-Schwimmer decomposition of 
${\bf Q}_4$.
\end{abstract}

\maketitle

\begin{center} \today \end{center}

\vspace{1cm}

\section{The results}\label{results}

The significance of the Yamabe operator
$$
    P_2 = \Delta - \left(\frac{n}{2}-1\right) \J
$$
and the Paneitz operator 
$$
   P_4 = \Delta^2 - \delta ((n-2) \J h - 4\Rho) d + \left(\frac{n}{2}-2\right) \left(\frac{n}{2} \J^2 - 2 |\Rho|^2 - \Delta (\J) \right)
$$
in geometric analysis is well-known (\cite{sharp}, \cite{CY}, \cite{CGY}, \cite{CBull}, \cite{DGH}, \cite{J1}, \cite{BJ} 
and references therein). These differential operators 
are defined on any Riemannian manifold $(M,g)$. Here we use the following conventions. Let $n$ be the dimension of $M$, 
$\delta$ the divergence operator on $1$-forms, $\Delta = \delta d$ the non-positive Laplacian, $2(n-1) \J = \scal$ 
and $(n-2)\Rho = \Ric - \J h$. $\Rho$ is the Schouten tensor of $h$. It naturally acts on $1$-forms. 
The operators $P_2$ and $P_4$ are the first two elements in the sequence of so-called GJMS-operators $P_{2N}$ \cite{GJMS}.
These self-adjoint geometric differential operators have leading term a power of the Laplacian $\Delta$ and are covariant
\begin{equation}\label{CTL-GJMS}
    e^{(\frac{n}{2}+N)\varphi} P_{2N} (\hat{h})(f) = P_{2N}(h)(e^{(\frac{n}{2}-N) \varphi} f)
\end{equation}
under conformal changes $h \mapsto e^{2\varphi} h$, $\varphi \in C^\infty(M)$, of the metric. 
The original definition of GJMS-operators rests on the ambient
metric of Fefferman and Graham \cite{FG-final}. The quantity $Q_4  = \frac{n}{2} \J^2 - |\Rho|^2 - \Delta (\J)$ is known as Branson's $Q$-curvature (of order $4$). There are analogous $Q$-curvatures $Q_{2N}$ associated to any GJMS-operator $P_{2N}$. For even $n$, the case $2N =n$ will be referred to as the critical case. A remarkable property of the critical $Q$-curvature $Q_n$ is its famous transformation law
$$
    e^{n\varphi} Q_n(\hat{h}) = Q_n(h) + P_n(h) (\varphi).
$$
This property may be derived as a consequence of the conformal covariance of $P_{2N}$ in the non-critical case using a continuation in dimension argument \cite{sharp}.

In recent works \cite[Section 8]{GW-LNY}, \cite{GW-RV}, \cite[Section 9]{JO1}, generalizations ${\bf P}_{N}$ 
of GJMS-operators and ${\bf Q}_N$ of Branson's $Q$-curvatures were introduced in the context of the singular Yamabe 
problem for hypersurfaces. These are called extrinsic conformal Laplacians and extrinsic $Q$-curvatures. Assume that the closed manifold $M$ is the boundary of a compact manifold $(X,g)$ with $h$ being induced by $g$. Let $\iota: M \hookrightarrow X$ denote the embedding. The operators ${\bf P}_N(g)$ still act on $C^\infty(M)$ and, for even $N$, have leading term a power of the Laplacian of $M$. But their lower-order terms depend on the embedding $\iota$. For odd $N$, their leading terms depend on the trace-free part $\lo = L - H h$ of the second fundamental form $L$. Here $H$ is the mean curvature, i.e., $\tr(L) = n H$. $M$ is called totally umbilic if $\lo=0$. Again, a basic property of the operators ${\bf P}_N$ is their covariance
$$
    e^{\frac{n+N}{2} \iota^*(\varphi)} {\bf P}_{N} (\hat{g})(f) = {\bf P}_{N}(g)(e^{\frac{n-N}{2} \varphi} f)
$$
under conformal changes $g \mapsto \hat{g} = e^{2\varphi} g$, $\varphi \in C^\infty(X)$. In contrast to 
\eqref{CTL-GJMS}, the latter property concerns conformal changes of the metric $g$ on the ambient space $X$. 
Since $\mathring{\hat{L}} = e^\varphi \lo$, the condition $\lo=0$ is conformally invariant.

In the following, extrinsic conformal Laplacians and extrinsic $Q$-curvatures will be denoted by boldface letters. 
For simplicity, we often omit their dependence on the metric. The operator ${\bf P}_1$ vanishes, and the first two 
non-trivial extrinsic conformal Laplacians are given by (see \cite[Proposition 8.5]{GW-LNY}, \cite[Sections 13.10-13.11]{JO1})

\begin{prop}\label{P23} It holds
\begin{equation}\label{ECL-2}
    {\bf P}_2(g) = P_2(h) + \frac{n-2}{4(n-1)} |\lo|^2, \quad n \ge 2
\end{equation}
and
\begin{equation}\label{ECL-3}
    {\bf P}_3(g) = 8 \delta (\lo d) + \frac{n-3}{2} \frac{4}{n-2} (\delta \delta (\lo) - (n-3) (\lo,\Rho) + (n-1) (\lo,\JF)), 
   \quad n \ge 3.
\end{equation}
Here $\JF = \iota^* (\bar{\Rho}) - \Rho + H \lo + \frac{1}{2} H^2 h$ is the conformally invariant Fialkow tensor. 
The corresponding extrinsic $Q$-curvatures are
$$
   {\bf Q}_2(g) = Q_2(h) + \frac{1}{2(n-1)}|\lo|^2 \quad \mbox{with} \quad Q_2(h) = \J_h
$$
and
$$
   {\bf Q}_3(g) = \frac{4}{n-2} (\delta \delta (\lo) - (n-3) (\lo,\Rho) + (n-1) (\lo,\JF)).
$$
These satisfy the fundamental transformation laws
$$
    e^{2 \iota^*(\varphi)} {\bf Q}_2(\hat{g}) = {\bf Q}_2(g) - {\bf P}_2(g)(\varphi)
$$
if $n=2$ and
$$
   e^{3 \iota^*(\varphi)} {\bf Q}_3(\hat{g}) = {\bf Q}_3(g) + {\bf P}_3(g)(\varphi)
$$
if $n=3$.
\end{prop}

The following theorem is the first new result in the present note. 
It provides explicit formulas for the extrinsic Paneitz operator ${\bf P}_4$ and the extrinsic $Q$-curvature ${\bf Q}_4$ 
in the case of a totally umbilic embedding $\iota$. Its formulation requires 
some more notation. Let $\overline{W}$ be the Weyl tensor of $g$ and let $\W_{ij} = \overline{W}_{0ij0}$ be defined by 
inserting a unit normal vector of $M$ into the first and the last slot of $\overline{W}$. Then $\W$ is a trace-free conformally 
invariant symmetric bilinear form on $M$. It acts naturally on $1$-forms on $M$. Sometimes, it will be convenient to denote 
the background metric $g$ by $\bar{g}$. The symbol $\delta$ also will be used for the divergence operator on symmetric bilinear forms.

\begin{theorem}\label{PQ4-gen} Assume that $\lo=0$ and $n \ge 4$. Then
\begin{align}\label{P4-final}
   {\bf P}_4 (f) & = P_4 (f) + 4 \frac{n-1}{n-2} \delta(\W df)  \notag\\
   & + \left(\frac{n}{2}-2\right) \frac{2(n-1)}{(n-2)(n-3)} \left(\frac{n-1}{n-2} |\W|^2 - (n-4) (\Rho,\W) 
   + \delta \delta (\W) \right) f
\end{align}
for $f \in C^\infty(M)$. In particular, it holds ${\bf P}_4 = P_4$ iff $\W=0$. As a consequence, we find
\begin{equation}\label{Q4-final}
   {\bf Q}_4 = Q_4 + \frac{2(n-1)}{(n-2)(n-3)} \left(\frac{n-1}{n-2} |\W|^2 - (n-4) (\Rho,\W) 
   + \delta \delta (\W) \right).
\end{equation}
Thus, in the critical dimension $n=4$, it holds 
\begin{equation}\label{P4-crit-0}
   {\bf P}_4(f) = P_4(f) + 6 \delta(\W df)
\end{equation}
and 
\begin{equation}\label{Q4-crit-0}
   {\bf Q}_4 = Q_4 + \frac{9}{2} |\W|^2 + 3 \delta \delta (\W).
\end{equation}
\end{theorem}

Some comments are in order. The assumption $\lo=0$ is a conformally invariant condition. In \cite{JO1}, a different 
normalization of ${\bf P}_4$ has been used. If $\bar{g} = dr^2 + h_r$ so that $g_+ = r^{-2} \bar{g}$ satisfies 
$\Ric (g_+) + n g_+ = 0$ (we refer to this as the Poincar\'e-Einstein case), then it holds $\lo=0$ and $\W = 0$. 
Hence, in this case, formula \eqref{P4-final} shows that ${\bf P}_4$ reduces to the Paneitz operator $P_4$. Formula 
\eqref{P4-final} makes the self-adjointness of ${\bf P}_4$ obvious. Formula \eqref{Q4-crit-0} implies that the total 
integral of ${\bf Q}_4$ equals
$$
     \int_{M^4} \left(Q_4 + \frac{9}{2} |\W|^2 \right) dvol_h 
$$
if $\lo=0$. Thus, the Gauss-Bonnet formula
$$
   8 \pi^2 \chi(M) = \frac{1}{4} \int_{M^4} |W|^2 dvol_h + \int_{M^4} Q_4 dvol_h
$$
implies that the total integral of ${\bf Q}_4$ is conformally invariant if $\lo=0$. One easily sees 
that $f \mapsto P_4(f) + 6 \delta(\W df)$ is a conformally covariant operator: both $P_4$ and $\delta(\W d)$ 
are conformally covariant. Also one can directly verify the fundamental transformation law 
\begin{equation}\label{ECTL}
    e^{4 \iota^*(\varphi)} {\bf Q}_4 (\hat{g}) = {\bf Q}_4(g) + {\bf P}_4(g)(\varphi)
\end{equation}
if $\lo=0$.

Next, we consider the general case but restrict to the critical dimension $n=4$. 

\begin{prop}\label{P4-crit} Let $n=4$. Then
\begin{equation}\label{EP4-crit}
    {\bf P}_4  = \Delta^2 - \delta (2 \J h - 4 \Rho) d + \delta \left(14 \lo^2 - \frac{4}{3} |\lo|^2 h + 6 \W \right) d.
\end{equation}
\end{prop}

Note that the right-hand side of \eqref{EP4-crit} is a sum of $P_4$ and three individually conformally covariant operators.

The conformal transformation property \eqref{ECTL} actually is true in general \cite[Section 10]{JO1}. Combining it with 
${\bf P}_4(1)=0$ and the self-adjointness of ${\bf P}_4$, shows that the total integral of ${\bf Q}_4$ always is a 
global conformal invariant. The next result describes this invariant in terms of the background metric. We continue to 
adopt the convention that curvature quantities of the background metric $g$ on $X$ are denoted by a bar in contrast to 
curvature quantities of the induced metric $h$ on $M$. The component $\bar{\Rho}_{00}$ is defined by inserting two 
unit normal vectors into the Schouten tensor $\bar{\Rho}$ of $g$.

\begin{theorem}\label{Q4-g-int} Let $n=4$. Then
\begin{align}\label{Q4-g-invariant}
     & \int_M {\bf Q}_4 dvol_h =  \int_M \left( 2 \J^2 - 2|\Rho|^2 + \frac{9}{2} |\W|^2\right) dvol_h \notag \\
     & + \int_M \left(2 (\lo,\bar{\nabla}_0 (\bar{\Rho})) 
     - 4 \lo^{ij} \bar{\nabla}_0 (\overline{W})_{0ij0} + 2 (\lo,\Hess(H)) + 2 H (\lo,\Rho) - 9 H (\lo,\W) \right) dvol_h \notag \\
     & + \int_M \left( 8 (\lo^2,\Rho) -2 \bar{\Rho}_{00} |\lo|^2 - 3 \J |\lo|^2  - 3 H^2 |\lo|^2 
    - H \tr(\lo^3) + 21 (\lo^2,\W) \right) dvol_h,
\end{align}
up to the integral of a linear combination of the local conformal invariants $|\lo|^4$ and $\tr(\lo^4)$. Here all scalar products and 
norms are defined by the metric $h$.
\end{theorem}


While the first integral on the right-hand side of \eqref{Q4-g-invariant} is independent of $\lo$, the second integral is linear in $\lo$, 
and all terms except the second-last last one in the third integral are quadratic in $\lo$.\footnote{We suppressed the terms which 
are quartic in $\lo$.}

Since the first integral in \eqref{Q4-g-invariant} itself is a global conformal invariant, the sum of the two remaining integrals 
also defines a global conformal invariant.\footnote{The term $(\lo^2,\W)$ is a local conformal invariant.} In order to 
understand that result better, we set
\begin{align}\label{new invariant}
      \mathcal{C} & \st 2 (\lo,\bar{\nabla}_0 (\bar{\Rho})) 
     - 4 \lo^{ij} \bar{\nabla}_0 (\overline{W})_{0ij0} + 2 (\lo,\Hess(H)) + 2 H (\lo,\Rho) - 9 H (\lo,\W) \notag \\
     & + 8 (\lo^2,\Rho) -2 \bar{\Rho}_{00} |\lo|^2 - 3 \J |\lo|^2  - 3 H^2 |\lo|^2 
    - H \tr(\lo^3) + 2 \delta \delta (\lo^2) + \frac{1}{2} \Delta(|\lo|^2).
\end{align}
All terms in that sum except the last two were taken from \eqref{Q4-g-invariant}. The advantage of adding these 
two divergence terms becomes clear in view of the following result.


\begin{prop}\label{LCI} Let $n=4$. Then 
$
    e^{4\varphi} \hat{\mathcal{C}} = \mathcal{C},
$
i.e., $\mathcal{C}$ is a local conformal invariant.
\end{prop}

These results imply the following decomposition result for the critical extrinsic $Q$-curvature of order $4$.

\begin{theorem}\label{alex} Let $n=4$. Then the critical extrinsic $Q$-curvature ${\bf Q}_4$ is a linear combination of
the Pfaffian of $M$, local conformal invariants of the embedding $M \hookrightarrow X$ and a divergence term.
\end{theorem}

This result is an analog of the Alexakis-Deser-Schwimmer decomposition of global conformal invariants (established by 
Alexakis in \cite{alex} and a series of papers). The local conformal invariants of relevance here are the trivial terms
$|\lo|^4$, $\tr(\lo^4)$, $(\lo^2,\W)$, $|\W|^2$, $|W|^2$ and the non-trivial term $\mathcal{C}$. For other results in 
this direction we refer to \cite{mondino}.

{\em Acknowledgment.} The author is grateful to B.  {\O}rsted for numerous discussions.

\section{Outline of proofs}\label{proofs}

We briefly describe the proofs of the results formulated in Section \ref{results}. The main 
focus will be on ${\bf P}_4$ and ${\bf Q}_4$. 

There are different methods to define extrinsic conformal Laplacians ${\bf P}_N$. In \cite{GW-LNY}, Gover and 
Waldron derived these operators from compositions of so-called Laplace-Robin operators. The latter notion has its origin 
in conformal tractor calculus. But Laplace-Robin operators actually are also linked to representation theory 
\cite{JO0} and scattering theory \cite{JO1}. From the point of view developed in \cite{JO1}, the extrinsic 
conformal Laplacians appear in terms of so-called residue families as introduced in \cite{J1} 
in the setting of Poincar\'e-Einstein metrics. This also provides a natural definition of extrinsic $Q$-curvatures. 
Roughly speaking, residue families may be viewed as curved versions of symmetry breaking operators in representation theory \cite{KS}. 
Now \cite[Theorem 4]{JO1} states that the extrinsic conformal Laplacians can be identified with residues of 
the geometric scattering operator of the singular metric $\sigma^{-2} g$, where $\sigma \in C^\infty(X)$ satisfies the 
condition
$$
   \scal_{\sigma^{-2} g} = - n(n+1)
$$
(at least asymptotically). In other words, $\sigma$ is a solution of a singular Yamabe problem. 
This result extends a result of \cite{GZ} for GJMS-operartors. Since the singular metric 
$\sigma^{-2}g$ is asymptotically hyperbolic, one still can build on results in scattering theory as developed in \cite{GZ}. 
A different perspective was taken in \cite{CMY} by {\em defining} extrinsic conformal Laplacians through the residues
of the scattering operator. However, this paper did not clarify the relation of these residues to the operators of Gover-Waldron.
The residues of the scattering operator of interest here can be described in terms of the asymptotic expansion of eigenfunctions 
of the Laplacian of $\sigma^{-2}g$. In order to analyze these expansions, one has to choose suitable coordinates. 
In \cite{JO1}, we utilized so-called adapted coordinates (which are best suited for the study of residue families). On the 
other hand, one also may use coordinates so that the metric $\sigma^{-2}g$ takes the form 
$\hat{r}^{-2}(d\hat{r}^2 + h_{\hat{r}})$. Then the metric $\hat{\bar{g}} \st d\hat{r}^2 + h_{\hat{r}}$ is 
conformally related to the original metric $\bar{g}$, i.e., it holds 
$$
   \hat{\bar{g}} = e^{2\omega} \bar{g}
$$
with some $\omega \in C^\infty(X)$ so that $\iota^*(\omega) = 0$.\footnote{These coordinates have been used 
recently in \cite{CMY} to derive formulas for ${\bf P}_2$ and ${\bf P}_3$ which are equivalent to those in Proposition \ref{P23}.} 

Let $\bar{\J}$ and $\bar{\G}_{ij} = \bar{R}_{0ij0}$ be defined for the metric $\bar{g}$, and let $\hat{\bar{\J}}$ and $\hat{\bar{\G}}$ denote these quantities for the metric $\hat{\bar{g}}$. Let $'$ denote the derivative in $\hat{r}$. In these terms, we use the asymptotic expansion of eigenfunctions of the Laplacian of the singular metric
$\hat{r}^{-2} \hat{\bar{g}}$ of constant scalar curvature $-n(n+1)$ into powers of $\hat{r}$ to read off that
\begin{align*}
   {\bf P}_2 = \Delta - \tfrac{n-2}{2} \hat{\bar{\J}} \quad \mbox{and} \quad {\bf P}_3 \sim 
   \Delta'-\tfrac{n-3}{2} \hat{\bar{\J}}',
\end{align*}
where $\Delta_{h_{\hat{r}}} = \Delta_h + \hat{r} \Delta'_h + \hat{r}^2 \Delta_h''+ \cdots$ and the sign $\sim$ 
indicates equality up to a constant multiple. In order to further evaluate these formulas, we apply 
the following result.  As before, constructions with a bar refer to the metric $\bar{g} = dr^2 + h_r$ and constructions 
without a bar refer to the metric $h = h_0$ on $M$. 

\begin{lem}\label{J-der} If $r^{-2}(dr^2 + h_r) $ has constant scalar curvature $-n(n+1)$, then $H = 0$, 
\begin{align*}
    \bar{\J} & = \J - \frac{1}{2(n-1)} |\lo|^2, \\
    (n-2) \bar{\J}' & = \delta \delta (\lo) + (\lo,\overline{\Ric}) - 2(\lo,\Ric)
\end{align*}
and 
\begin{align*}
   (n-3) \bar{\J}'' & = - 3 \bar{\J}^2 - 2 (\lo, \nabla (\overline{\Ric}_0)) + \delta (\bar{\nabla}_0(\overline{\Ric})_{0}) 
    - (\lo, \bar{\nabla}_0(\overline{\Ric})) \notag \\
    & - 2 (\delta(\lo),\overline{\Ric}_0) + \delta ((\lo \overline{\Ric})_{0}) 
   + (\lo^2, \overline{\Ric}) + (\bar{\G}, \overline{\Ric}) + 2 \lo^{ij} \bar{\nabla}_0 (\bar{R})_{0ij0} \notag \\
   & + 2 (\overline{\Ric}_{00})^2 - 2 |\bar{\G}|^2 - 8 (\lo^2,\bar{\G}) + 5 |\lo|^2 \overline{\Ric}_{00} + 24 \sigma_4(\lo).
\end{align*}
Here $\sigma_4(\lo)$ is the fourth elementary symmetric function in the eigenvalues of $\lo$. In particular,
\begin{align}\label{S0}
   \bar{\J} & = \J, \notag \\
   \bar{\J}' & = 0, \notag \\
    (n-3) \bar{\J}'' & = - 3 \J^2 + (\bar{\G}, \overline{\Ric}) + 2 (\overline{\Ric}_{00})^2 
   - 2 |\bar{\G}|^2 + \delta (\bar{\nabla}_0(\overline{\Ric})_{0})
\end{align}
if $\lo=0$.
\end{lem}

\begin{proof} The assumption implies the basic relation
\begin{equation}\label{key}
    -r \bar{\J} = \frac{v'}{v},
\end{equation}
where $v(r) = dvol_{h_r}/dvol_h = 1+ r v_1 + r^2 v_2 + \cdots$. Now $v_1 = n H$ implies $H=0$. Combining 
the relation \eqref{key} with known general formulas for the coefficients $v_k$ and formulas for the first two 
normal derivatives of the Einstein tensor of $\bar{g}$ (generalizing results in \cite{JO2}) proves the assertions.
\end{proof}

\begin{cor}\label{int} If $r^{-2}(dr^2 + h_r) $ has constant scalar curvature $-n(n+1)$, then
\begin{multline*}
     (n-3) \int_M \bar{\J}'' = \int_M -3 \bar{\J}^2 + (\bar{\G}, \overline{\Ric}) + 2 (\overline{\Ric}_{00})^2 - 2 |\bar{\G}|^2 \\
     + \int (\lo^2,\overline{\Ric}) + 5 |\lo|^2 \overline{\Ric}_{00} - 8 (\lo^2,\bar{\G}) + 24 \sigma_4(\lo) 
     + 2 \lo^{ij} \bar{\nabla}_0 (\bar{R})_{0ij0} - (\lo, \bar{\nabla}_0(\overline{\Ric})).
\end{multline*}
In particular, we get
$$
   (n-3)  \int_{M} \bar{\J}'' = \int_{M} -3 \J^2 + (\bar{\G}, \overline{\Ric}) + 2 (\overline{\Ric}_{00})^2 
   - 2 |\bar{\G}|^2
$$
if $\lo=0$.
\end{cor}

We emphasize that $\bar{\J}''$ and its integral substantially simplify under the assumption $\lo=0$. 

The first identity in  Lemma \ref{J-der} gives
$$
  {\bf P}_2 = \Delta - \frac{n-2}{2} \left(\J - \frac{1}{2(n-1)}|\lo|^2 \right)
$$
proving \eqref{ECL-2}. The conformal transformation law for the Ricci tensor implies
$$
   \hat{\overline{\Ric}}_{ij} = \overline{\Ric}_{ij} - (n-1) \overline{\Hess}_{ij}(\omega) - \bar{\Delta}(\omega) h_{ij}
   = \overline{\Ric}_{ij} - (n-1) (\Hess(\omega)  + L_{ij} \partial_0(\omega)) - \bar{\Delta}(\omega) h_{ij}.
$$ 
Thus, using $\partial_0(\omega) = - H$, we get
$$
   \lo^{ij} \hat{\overline{\Ric}}_{ij} = \lo^{ij} \overline{\Ric}_{ij} + (n-1) H |\lo|^2.
$$
But 
$$
   \Delta' = [\delta,\mathcal{H}_1]d = \delta (\mathcal{H}_1 d) - \mathcal{H}_1 \delta d,
$$ 
where
$$
\begin{cases}
   \mathcal{H}_1 = v_1 = 0 & \mbox{on $\Omega^0(M)$} \\
   \mathcal{H}_1 = - h_{(1)} + v_1 \id = - 2 \lo & \mbox{on $\Omega^1(M)$},
\end{cases}
$$
and $h_r = h + r h_{(1)} + r^2 h_{(2)} + \cdots$. By $v_1 = n H = 0$, we find $\Delta' = - 2 \delta (\lo d)$. 
Thus, the second identity in Lemma \ref{J-der} yields
$$
   {\bf P}_3 \sim - \delta (\lo d) 
   - \frac{n-3}{4(n-2)} (\delta \delta (\lo) + (\lo,\overline{\Ric}) - 2(\lo,\Ric)+ (n-1) H |\lo|^2).
$$
This implies \eqref{ECL-3}. 

We continue with the discussion of ${\bf P}_4$. The following unconditional result is the next main result. 

\begin{theorem}\label{PQ4-ext} The extrinsic Paneitz operator is given by 
\begin{equation}\label{P4-gen}
    {\bf P}_4  = \Delta^2 - \delta ((n-2)\hat{\bar{\J}} h + 4 \hat{h}_{(2)}) d  + 4 \delta \hat{h}_{(1)}^2 d 
    + \left(\frac{n}{2}-2\right) {\bf Q}_4,
\end{equation}
where $\hat{h}_{(1)} = \lo$, $\hat{h}_{(2)} = \lo^2 - \hat{\bar{\G}}$ and 
\begin{equation}\label{Q4-gen}
   { \bf Q}_4 = \frac{n}{2} \hat{\bar{\J}}^2 - 2 \hat{\bar{\J}}'' - \Delta (\hat{\bar{\J}}).
\end{equation}
In particular, ${\bf P}_4$ is self-adjoint.
\end{theorem}

\begin{proof} A calculation of the expansion of eigenfunctions of the Laplacian of the singular metric 
$\hat{r}^{-2} (d\hat{r}^2 + h_{\hat{r}})$ of constant scalar curvature $-n(n+1)$ into powers of $\hat{r}$ shows that
$$
   {\bf P}_4 = \left(\Delta - \frac{n}{2} \hat{\bar{\J}}\right) \left(\Delta - \left(\frac{n}{2}-2\right) \hat{\bar{\J}}\right) 
   + 4 \hat{\Delta}'' - (n-4) \hat{\bar{\J}}''.
$$
Now we apply the formula
$
    \Delta'' = ([\delta,\mathcal{H}_2] - \mathcal{H}_1 [\delta,\mathcal{H}_1])d
$
for the second metric variation of the Laplacian. Here 
$$
\begin{cases}
    \mathcal{H}_2 = v_2     & \mbox{on  } \Omega^0(M), \\
    \mathcal{H}_2 = v_2 \id - v_1 h_{(1)} + (h_{(1)}^2 - h_{(2)}) & \mbox{on $\Omega^1(M)$} .
\end{cases}
$$
By $v_1 = 0$, the variation formula simplifies to
$
    \Delta'' = [\delta,\mathcal{H}_2] d
$
with $\mathcal{H}_2 = v_2 \id + (h_{(1)}^2 - h_{(2)})$. But $v_2 = -1/2 \bar{\J}$. Applying these results for the 
metric $\hat{\bar{g}}$ gives the claimed formula.
\end{proof}

The second-order terms in \eqref{P4-gen} easily can be made more explicit. 

\begin{lem} It holds
\begin{align*}
   {\bf P}_4 & = \Delta^2 - \delta((n-2)\J h - 4 \Rho) d + 4 \frac{n-1}{n-2} \delta \W  d \\
   & + \delta \left(4 \frac{3n-5}{n-2} \lo^2 + \frac{n^2-12n+16}{2(n-1)(n-2)} |\lo|^2 h \right) d 
   +  \left(\frac{n}{2}-2\right)  {\bf Q}_4.
\end{align*}
\end{lem}

In particular, in the critical dimension $n=4$, this implies Proposition \ref{P4-crit}.

In the Poincar\'e-Einstein case, it holds $\lo=0$, $\W=0$ and $\hat{\bar{\J}} = \J$, $\hat{\bar{\J}}'' = |\Rho|^2$ on $M$. 
In particular, we find ${\bf P}_4 = P_4$ and ${\bf Q}_4 = Q_4$.

Now let $n=4$. By \eqref{ECTL}, ${\bf P}_4(1)=0$ and the self-adjointness of ${\bf P}_4$, the total  integral 
of ${\bf Q}_4$ is a global conformal invariant. Hence Theorem \ref{PQ4-ext} implies

\begin{cor} Let $n=4$. Then
\begin{equation}\label{total}
    \int_{M} (\hat{\bar{\J}}^2 - \hat{\bar{\J}}'') dvol_h
\end{equation}
is a global conformal invariant of the background metric.
\end{cor}

This result generalizes the well-known fact that, in dimension $n=4$, the integral $\int \J^2 - |\Rho|^2$ is a global  
conformal invariant. Theorem \ref{Q4-g-int} makes that invariant explicit in the metric $g$. An essential ingredient 
in its proof is Corollary \ref{int}. We note that an alternative proof of Theorem \ref{Q4-g-int} 
combines the relation between the total integral of ${\bf Q}_4$ and the singular Yamabe energy (as defined in 
\cite{G-SY}) in dimensions $n=4$ with an evaluation of the relevant energy functional. Proposition \ref{LCI} follows  
by a direct calculation of the conformal variation of $\mathcal{C}$, and Theorem \ref{alex} 
is a direct consequence of Theorem \ref{Q4-g-int} and Proposition \ref{LCI}.

Now we return to the case $\lo=0$. In order to derive Theorem \ref{PQ4-gen} from Theorem \ref{PQ4-ext}, 
it remains to express the curvature data of the metric $\hat{\bar{g}}$ in terms of curvature data of the original 
metric $\bar{g}$. The third part of Lemma \ref{J-der} yields

\begin{lem}\label{J-N2} Assume that $r^{-2}(dr^2 + h_r) = r^{-2} \bar{g}$ has constant scalar 
curvature $-n(n+1)$ and that $\lo=0$. Then
\begin{align*}
    (n-3) \bar{\J}'' = (n-3) |\Rho|^2 - \frac{(n-1)^2}{(n-2)^2} |\W|^2 + \frac{(n-4)(n-1)}{(n-2)} (\Rho,\W) 
   + \delta (\bar{\nabla}_0(\overline{\Ric})_{0}).
\end{align*}
\end{lem}

We apply this result for the metric $\hat{\bar{g}}$. It remains to discuss the divergence term 
$\delta (\hat{\bar{\nabla}}_0(\hat{\overline{\Ric}})_{0})$. The properties $\omega = 0$ on $M$ and 
\begin{itemize}
\item $\partial_0(\omega) = - H$,
\item $\partial_0^2(\omega) = (n+1)/2 H^2 + \bar{\J} - \J$
\end{itemize}
on $M$ (if $\lo=0$) \cite[Lemma 5.4]{CMY} yield

\begin{lem}\label{div-ex} Assume that $\lo=0$. Then 
\begin{equation*}\label{trans}
    \delta (\hat{\bar{\nabla}}_0(\hat{\bar{\Rho}})_0) = \delta(\bar{\nabla}_0(\bar{\Rho})_0) 
    - \Delta (\bar{\Rho}_{00}+ H^2).
\end{equation*}
\end{lem}

The following result further simplifies this term.

\begin{lem}\label{simple} Assume that $\lo=0$. Then 
$$
    \delta (\bar{\nabla}_0(\bar{\Rho})_0) - \Delta (\bar{\Rho}_{00} + H^2) = - \frac{1}{n-2} \delta \delta (\W).
$$
\end{lem}

Combining the above results, completes the proof of Theorem \ref{PQ4-gen}.

Detailed proofs will appear in a forthcoming paper.


\end{document}